\documentclass[10pt,reqno]{amsart}

\usepackage{amsmath,amssymb,mathtools}
\usepackage{microtype}
\usepackage[hidelinks]{hyperref}

\newtheorem{theorem}{Theorem}[section]
\newtheorem{proposition}[theorem]{Proposition}
\newtheorem{lemma}[theorem]{Lemma}
\newtheorem{corollary}[theorem]{Corollary}
\theoremstyle{remark}
\newtheorem{remark}[theorem]{Remark}
\numberwithin{equation}{section}

\newcommand{\C}{\mathbb C}
\newcommand{\Q}{\mathbb Q}
\newcommand{\Z}{\mathbb Z}
\newcommand{\ip}[2]{\left\langle #1,#2\right\rangle}
\newcommand{\norm}[1]{\left\lVert #1\right\rVert}

\title[Explicit full-spark Gabor windows]
{Explicit Full-Spark Gabor Windows with Quantitative Phase Retrieval}

\author{Dongwei Li}
\address{School of Mathematics, Hefei University of Technology, Hefei, China}
\email{dongweili@hfut.edu.cn}


\begin{document}
\raggedbottom

\begin{abstract}
We construct explicit algebraic finite Gabor windows that are simultaneously
full spark and phase retrievable, with quantitative control of the
self-ambiguity function.  In every cyclic dimension, a constant--Chu window
has an exactly computable ambiguity minimum of order $N^{-3/2}$; for
$N\ge25$, a two-site Chu repair has margin comparable to $dN^{-3/2}$ for
every divisor $d\le\sqrt N$.  A quantitative algebraic regularization,
based on the high-degree specialization principle used in explicit full-spark
Gabor constructions, imposes full spark while retaining a fixed proportion
of a seed's ambiguity margin.  These ingredients give explicit simultaneous
constructions in every cyclic dimension and a factor-sensitive improvement
when $N$ has a divisor near $\sqrt N$.  Further consequences include exact Chinese-remainder tensorization,
optimal-order squarefree families, near-SIC prime regularization, lifted
stability, and a finite cyclic Schr\"odinger construction.
\end{abstract}

\maketitle
\enlargethispage{3pt}

\section{Introduction}

Let $N\ge2$, write $\omega_N=e^{2\pi i/N}$, and define the cyclic
translation and modulation operators on $\C^N$ by
\[
 (T_kf)(j)=f(j-k),\qquad (M_\ell f)(j)=\omega_N^{\ell j}f(j),
 \qquad j,k,\ell\in\Z_N.
\]
The finite Gabor orbit generated by $g\in\C^N$ is
\[
 \mathcal G(g)=\{M_\ell T_kg:(k,\ell)\in\Z_N^2\}.
\]
It is \emph{full spark} if every $N$-element subfamily of its $N^2$ vectors
forms a basis of $\C^N$.  It performs \emph{phase retrieval} if the intensities
\[
 x\longmapsto
 \bigl(|\ip{x}{M_\ell T_kg}|^2\bigr)_{k,\ell}
\]
determine $x$ up to a unimodular scalar.

The two conditions come from different parts of finite frame theory.
Full spark is the maximal linear-independence condition for a finite frame and
is closely tied to robustness under erasures; see, for example,
\cite{Li2026} for recent robustness characterizations and deterministic
full-spark constructions.  In the structured Gabor setting,
prime-dimensional linear-independence results go back to
Lawrence--Pfander--Walnut \cite{LawrencePfanderWalnut2005}, and Malikiosis
proved the existence of full-spark cyclic Gabor windows in every finite
dimension and supplied explicit algebraic constructions
\cite{Malikiosis2015}.  Bojarovska and Flinth showed that nonvanishing of the
self-ambiguity function
\begin{equation}
 A_g(k,\ell)=\ip{g}{M_\ell T_kg}
 =\sum_{j\in\Z_N}\overline{g(j)}\,\omega_N^{\ell j}g(j-k)
 \label{eq:ambiguity}
\end{equation}
is sufficient for finite Gabor phase retrieval
\cite[Theorem~2.2]{BojarovskaFlinth2016}; for broader background on
uniqueness and stability in phase retrieval, see
\cite{GrohsKoppensteinerRathmair2020}.  F\"uhr and Oussa later isolated,
in the finite Schr\"odinger representation, the explicit simultaneous
realization of phase retrieval and full spark as an open construction problem
arising in their inductive treatment of finite nilpotent groups
\cite[Concluding Remarks]{FuehrOussa2023}.

The central observation here is that full spark and quantitative phase
retrieval need not be produced by the same seed.  The determinant part of our
argument is a quantitative repurposing of the high-degree algebraic
specialization used by Malikiosis in explicit full-spark Gabor constructions
\cite[Corollary~5.2]{Malikiosis2015}.  Instead of using the specialization
only to construct a full-spark window, we start from an arbitrary algebraic
window $b$ with a controlled ambiguity margin and perturb it with a separately
chosen full-spark window $h$:
\[
 g=b+\varepsilon h.
\]
The parameter $\varepsilon$ has algebraic degree greater than $N$ over a
number field containing $\omega_N$ and the coordinates of $b$ and $h$, so no
maximal Gabor determinant can vanish; at the same time, $\varepsilon$ is made
small enough that a fixed proportion of the ambiguity margin of $b$ survives.
Thus the new point of the regularization theorem is the quantitative transfer
from an arbitrary algebraic ambiguity seed to a simultaneous full-spark
window, rather than the high-degree specialization principle by itself.
Generic coexistence of the two nonvanishing properties is not the issue here:
the aim is to specify deterministic algebraic windows and retain explicit
quantitative lower bounds for every ambiguity coefficient.

For a nonzero window $g$, put
\begin{equation}
 \alpha(g)=\min_{(k,\ell)\in\Z_N^2}
 \frac{|A_g(k,\ell)|}{\norm g_2^2}.
 \label{eq:alpha}
\end{equation}

\begin{theorem}[Main quantitative simultaneous constructions]
\label{thm:main-intro}
The following deterministic constructions hold.
\begin{enumerate}
\item For every $N\ge2$ there is an explicit algebraic window $H_N\in\C^N$
whose Gabor system is full spark and phase retrievable and satisfies
\[
 \alpha(H_N)\ge \frac{47}{162}N^{-3/2}.
\]
\item For $N\ge25$, let
\[
 d_*(N)=\max\{d:d\mid N,\ d\le\sqrt N\}.
\]
There is an explicit algebraic window $G_N\in\C^N$ whose Gabor system is
full spark and phase retrievable and satisfies
\[
 \alpha(G_N)\ge \frac{47}{648}
 \frac{d_*(N)}{N^{3/2}}.
\]
\end{enumerate}
The underlying seed statements are sharper: Theorem~\ref{thm:constant-chu}
gives an exact ambiguity minimum for the universal constant--Chu seed, and
Theorem~\ref{thm:divisor-chu} proves a two-sided divisor-sensitive law.
\end{theorem}

\begin{remark}[Choosing the better universal or arithmetic seed]
\label{rem:combined-main}
For $N\ge25$, choose between the two windows in
Theorem~\ref{thm:main-intro} according to which certified bound is larger.
This gives an explicit algebraic full-spark phase-retrieval window $F_N$ with
\[
 \alpha(F_N)\ge \frac{47}{648}
 \frac{\max\{4,d_*(N)\}}{N^{3/2}}.
\]
Thus the two-site repair gives a genuine improvement over the universal
constant--Chu guarantee precisely when $d_*(N)>4$.
\end{remark}

The two Chu constructions below are the main constructive results.  The first quantitative seed is available in every cyclic dimension.  Periodic
ambiguity functions of CAZAC and chirp waveforms have long been studied; see,
for example, Benedetto and Donatelli \cite{BenedettoDonatelli2007}.  Following Chu's original polyphase construction \cite{Chu1972}, our use
of the Chu chirp is different: we seek a positive \emph{global} ambiguity
floor after an explicit nonunimodular repair.  A constant plus a
parity-adjusted Chu chirp has an exactly computable ambiguity margin of order
$N^{-3/2}$.  A second two-site Chu construction detects the arithmetic of
$N$: with $d_*(N)$ as in Theorem~\ref{thm:main-intro}, for $N\ge25$ the
two-site seed satisfies
\[
 \alpha\bigl(b_{N,d_*(N)}\bigr)\asymp \frac{d_*(N)}{N^{3/2}}.
\]
Thus a divisor near $\sqrt N$ produces a genuine gain; for example, square
dimensions have seed margin of order $N^{-1}$.  Algebraic regularization
retains the corresponding lower bound up to an absolute factor.

It is useful to place these bounds on the projector-Gram scale
$\lambda_{\min}=N\alpha(g)^2$ used for Weyl--Heisenberg measurements.
The universal construction gives $\lambda_{\min}\gtrsim N^{-2}$ in every
cyclic dimension, the divisor-sensitive construction gives
$\lambda_{\min}\gtrsim d_*(N)^2N^{-2}$, and square dimensions therefore
reach the scale $N^{-1}$.  With a bounded number of prime factors at least
$5$, the squarefree CRT construction below gives $\lambda_{\min}\gtrsim1$.
For comparison, Zhu and Wang's explicit all-dimensional cyclic family has
projector-Gram floor $\Theta(N^{-3})$ in odd dimensions and
$\Theta(N^{-5})$ in even dimensions \cite[Proposition~7]{ZhuWang2026}.
Their uniform prime-power finite-field constructions for nonprime dimensions
use the different phase space $\mathbb F_q^2$, rather than the cyclic phase
space $\Z_{p^r}^2$.  This comparison is only between the stated explicit
families and does not assert a global cyclic max--min optimum.

Further consequences include exact Chinese-remainder tensorization,
optimal-order squarefree families, and stability estimates for the complete
$N^2$-measurement operator.  Combining CRT with the balanced-Alltop prime
seeds of Zhu and Wang \cite{ZhuWang2026} recovers the optimal
$N^{-1/2}$ ambiguity scale on suitable squarefree dimensions; in prime
dimension, the same regularization enforces full spark while changing their
asymptotically SIC-flat nonidentity spectrum only by a lower-order amount.
The near-SIC property itself is due to Zhu and Wang.

The paper is organized accordingly.  Section~\ref{sec:regularization}
proves the regularization principle and records the ambiguity/spectral
interface.  Section~\ref{sec:constructions} gives the universal Chu seed, the
divisor-sensitive refinement, Chinese-remainder factorization, and the prime
near-SIC corollary.  Section~\ref{sec:consequences} derives stability and the
finite Schr\"odinger consequence.

\section{Algebraic regularization and the spectral interface}
\label{sec:regularization}

We first separate the determinant and ambiguity requirements.

\begin{theorem}[Quantitative algebraic full-spark regularization]
\label{thm:regularization}
Let $b,h\in\C^N$ be unit vectors whose coordinates are algebraic, and assume
that $\mathcal G(h)$ is full spark.  Let $K$ be a number field containing $\omega_N$ and
the coordinates of $b$ and $h$.  Choose a prime $q>N$ and a rational prime
$r$ unramified in $K$, and set
\[
 \varepsilon=r^{-1/q},\qquad g=b+\varepsilon h.
\]
Then $\mathcal G(g)$ is full spark.  Moreover,
\begin{equation}
 \alpha(g)\ge
 \frac{\alpha(b)-2\varepsilon-\varepsilon^2}{(1+\varepsilon)^2}.
 \label{eq:regularization-bound}
\end{equation}
In particular, if $\varepsilon\le\alpha(b)/8$, then
\begin{equation}
 \alpha(g)\ge\frac{47}{81}\alpha(b).
 \label{eq:regularization-47}
\end{equation}
\end{theorem}

\begin{proof}
Because $r$ is unramified in $K$, every prime ideal above $r$ occurs to
first order in $r\mathcal O_K$.  Eisenstein's criterion at any such prime
ideal proves that $X^q-r$ is irreducible over $K$.  Hence
$[K(\varepsilon):K]=q$.

Fix an $N$-element subset $\Lambda\subset\Z_N^2$, and let
$D_\Lambda(z)$ denote the determinant of the corresponding $N\times N$
Gabor submatrix generated by a formal window $z$.  Then
\[
 P_\Lambda(t)=D_\Lambda(b+th)\in K[t]
\]
has degree at most $N$.  Its coefficient of $t^N$ is $D_\Lambda(h)$,
which is nonzero by full spark.  Thus $P_\Lambda$ is nonzero, and since
$\deg P_\Lambda<q$, we have $P_\Lambda(\varepsilon)\neq0$.  This holds for
every $\Lambda$, so $\mathcal G(g)$ is full spark.

Expanding the ambiguity function gives, uniformly in $(k,\ell)$,
\[
 |A_g(k,\ell)-A_b(k,\ell)|
 \le 2\varepsilon+\varepsilon^2,
\]
because $b$ and $h$ are unit vectors.  Also
$\norm g_2\le1+\varepsilon$.  This proves
\eqref{eq:regularization-bound}.  Since $\alpha(b)\le1$ and
$\varepsilon\le\alpha(b)/8$,
\[
 2\varepsilon+\varepsilon^2
 \le \frac14\alpha(b)+\frac1{64}\alpha(b)
 =\frac{17}{64}\alpha(b),
\]
whereas $(1+\varepsilon)^2\le(9/8)^2=81/64$.  Hence
\eqref{eq:regularization-47} follows.
\end{proof}

\begin{remark}[A fixed explicit full-spark regularizer]
\label{rem:explicit-regularizer}
All choices above are effective.  We also fix a deterministic full-spark
regularizer in every dimension.  Put
\[
 M_N=N(N-1)^2+1,
\]
let $s_N$ be the least rational prime not dividing $N$, and let
$\xi_N=s_N^{1/M_N}$ be the positive real root.  Since $s_N$ is unramified
in $\Q(\omega_N)$, Eisenstein at a prime above $s_N$ gives
$[\Q(\omega_N,\xi_N):\Q(\omega_N)]=M_N$.  Malikiosis'
one-parameter specialization then shows that
\[
 h_N(j)=\xi_N^{j^2},\qquad 0\le j<N,
\]
generates a full-spark Gabor system; see
\cite[Corollary~5.2]{Malikiosis2015}.  We use its unit normalization whenever
Theorem~\ref{thm:regularization} requires a full-spark seed.  The standard
cyclotomic ramification fact used here is recorded, for example, in
\cite{Washington1997}.

Once $b$ and $h_N$ are fixed, one may take the least prime $q>N$ and then
the least rational prime $r$ unramified in the resulting number field and
large enough to meet any prescribed upper bound on
$\varepsilon=r^{-1/q}$.
\end{remark}

\begin{remark}[Deterministic window-selection rule]
\label{rem:deterministic-rule}
Given $N$ and one of the algebraic ambiguity seeds used below, normalize that
seed and the fixed window $h_N$ from Remark~\ref{rem:explicit-regularizer}.
Fix any explicit certified lower bound $0<\beta\le\alpha(b)$.  Let $K$ be
the number field generated by $\omega_N$ and the coordinates of the two
normalized windows, let $q$ be the least prime greater than $N$, and let $r$
be the least rational prime not dividing the discriminant of $K$ and satisfying
\[
 r\ge \left(\frac{8}{\beta}\right)^q.
\]
Then $r$ is unramified in $K$,
$\varepsilon=r^{-1/q}\le\beta/8\le\alpha(b)/8$, and
$g=b+\varepsilon h_N$ is the explicit regularized window.  No prior
computation of the exact ambiguity minimum is required.  For the universal
constant--Chu seed one may take
$\beta_N=1/(2N^{3/2})$, while for the two-site seed one may take
$\beta_{N,d}=d/(8N^{3/2})$.  For CRT products, the product of the local
certified lower bounds is admissible.  Thus the word ``explicit'' requires no
generic choice: every ingredient can be selected by a fixed deterministic
rule.
\end{remark}

The same ambiguity coefficients control the complete lifted measurement
spectrum.  For a unit vector $g$, define
\begin{equation}
 \mathcal Q_g(X)(k,\ell)=
 \operatorname{tr}\!\left(X(M_\ell T_kg)(M_\ell T_kg)^*\right),
 \qquad X\in\C^{N\times N}.
 \label{eq:Qg}
\end{equation}

\begin{proposition}[Ambiguity diagonalization]
\label{prop:diagonalization}
Let $g$ be unit norm.  If $\widehat X(k,\ell)$ are the coefficients of $X$
in the orthonormal Weyl basis $N^{-1/2}M_\ell T_k$, then
\begin{equation}
 \norm{\mathcal Q_g(X)}_2^2
 =N\sum_{k,\ell}|A_g(k,\ell)|^2|\widehat X(k,\ell)|^2.
 \label{eq:diag}
\end{equation}
Consequently,
\begin{equation}
 \sqrt N\,\alpha(g)\norm X_{\mathrm{HS}}
 \le \norm{\mathcal Q_g(X)}_2
 \le \sqrt N\norm X_{\mathrm{HS}}.
 \label{eq:lifted-bounds}
\end{equation}
If
\[
 \Pi_{k,\ell}=(M_\ell T_kg)(M_\ell T_kg)^*,
 \qquad
 G_g^{\Pi}[(k,\ell),(k',\ell')]
   =\operatorname{tr}(\Pi_{k,\ell}\Pi_{k',\ell'}),
\]
then the Hilbert--Schmidt projector Gram matrix $G_g^{\Pi}$ has eigenvalues
\begin{equation}
 \{N|A_g(k,\ell)|^2:(k,\ell)\in\Z_N^2\},
 \label{eq:projector-spectrum}
\end{equation}
up to the standard symplectic relabeling.
\end{proposition}

\begin{proof}
Write $W_{k,\ell}=M_\ell T_k$.  The operators
$N^{-1/2}W_{k,\ell}$ form an orthonormal Hilbert--Schmidt basis.  If
$c_{k,\ell}=\operatorname{tr}(W_{k,\ell}^*X)$, then
\[
 X=\frac1N\sum_{k,\ell}c_{k,\ell}W_{k,\ell},
 \qquad
 \norm X_{\mathrm{HS}}^2=\frac1N\sum_{k,\ell}|c_{k,\ell}|^2.
\]
Likewise, the Weyl coefficient of $gg^*$ at $(k,\ell)$ has modulus
$|A_g(k,\ell)|$.  Conjugation by $W_{a,b}$ multiplies each Weyl basis
vector by the corresponding phase-space character.  Hence the unitary
Fourier transform on $\Z_N^2$ sends $\mathcal Q_g(X)$, up to a fixed
symplectic permutation and unimodular factors, to the array
\[
 (k,\ell)\longmapsto c_{k,\ell}A_g(k,\ell).
\]
Parseval and the preceding normalization give \eqref{eq:diag}.  The
inequalities follow from
$\alpha(g)\le |A_g(k,\ell)|\le1$.  Applying the same Fourier
characters to the phase-space circulant matrix $G_g^{\Pi}$ gives the
eigenvalues in \eqref{eq:projector-spectrum}; see also
\cite{GoldbergerKangOkoudjou2022}.  The corresponding Gabor stability
interface is discussed in \cite{AlaifariWellershoff2021}.
\end{proof}

\begin{lemma}[Moyal barrier]
\label{lem:moyal}
Every unit vector $g\in\C^N$ satisfies
\begin{equation}
 \alpha(g)\le\frac1{\sqrt{N+1}}.
 \label{eq:moyal-barrier}
\end{equation}
\end{lemma}

\begin{proof}
The finite Moyal identity gives
$\sum_{k,\ell}|A_g(k,\ell)|^2=N$.  Since $A_g(0,0)=1$,
\[
 N\ge1+(N^2-1)\alpha(g)^2,
\]
which is equivalent to \eqref{eq:moyal-barrier}.
\end{proof}

\section{Quantitative cyclic constructions}
\label{sec:constructions}

\subsection{A universal constant--Chu seed}

Define the parity-adjusted Chu chirp (in the normalization adapted to our
translation--modulation convention) \cite{Chu1972}
\begin{equation}
 q_N(j)=
 \begin{cases}
  \exp(\pi i j^2/N),&N\text{ even},\\[1mm]
  \exp(\pi i j(j+1)/N),&N\text{ odd}.
 \end{cases}
 \label{eq:chu}
\end{equation}
It is $N$-periodic and unimodular.  Put
\[
 Q_N(\ell)=\sum_{j\in\Z_N}q_N(j)\omega_N^{\ell j}.
\]
A shift of the summation variable gives
\begin{equation}
 Q_N(\ell)=
 \begin{cases}
  e^{-\pi i\ell^2/N}Q_N(0),&N\text{ even},\\
  e^{-\pi i\ell(\ell+1)/N}Q_N(0),&N\text{ odd}.
 \end{cases}
 \label{eq:chu-fourier}
\end{equation}
Parseval then implies $|Q_N(\ell)|=\sqrt N$ for every $\ell$; in
particular $Q_N(0)\neq0$.  Write
\[
 \eta_N=Q_N(0)/\sqrt N,
 \qquad
 \tau_N=
 \begin{cases}
  e^{\pi i/(2N)},&N\text{ even},\\
  1,&N\text{ odd},
 \end{cases}
 \qquad
 \lambda_N=\tau_N\eta_N^{-1}.
\]

\begin{theorem}[Exact universal Chu margin]
\label{thm:constant-chu}
Let
\begin{equation}
 b_N(j)=1+\lambda_Nq_N(j).
 \label{eq:constant-chu-window}
\end{equation}
Then
\begin{equation}
 \alpha(b_N)=
 \frac{\sqrt N\sin(\pi/(2N))}
 {N+\sqrt N\,\operatorname{Re}\tau_N}.
 \label{eq:constant-chu-exact}
\end{equation}
In particular,
\begin{equation}
 \alpha(b_N)\ge\frac1{2N^{3/2}}
 \qquad(N\ge2),
 \label{eq:universal-seed-lower}
\end{equation}
and $\alpha(b_N)\sim (\pi/2)N^{-3/2}$.
\end{theorem}

\begin{proof}
The chirp satisfies
\begin{equation}
 q_N(j-k)=a_kq_N(j)\omega_N^{-kj},\qquad |a_k|=1,
 \label{eq:chu-shift}
\end{equation}
where
\[
 a_k=
 \begin{cases}
 e^{\pi i k^2/N},&N\text{ even},\\
 e^{\pi i(k^2-k)/N},&N\text{ odd}.
 \end{cases}
\]
Hence
\[
 A_{q_N}(k,\ell)=Na_k\,\mathbf 1_{\ell=k}.
\]
Since $|\lambda_N|=1$, direct expansion gives the useful identity
\begin{equation}
 A_{b_N}(k,\ell)
 =N\mathbf 1_{\ell=0}+Na_k\mathbf 1_{\ell=k}
  +\lambda_Na_kQ_N(\ell-k)
  +\overline{\lambda_N}\,\overline{Q_N(-\ell)}.
 \label{eq:constant-chu-expansion}
\end{equation}
Using \eqref{eq:chu-fourier} in \eqref{eq:constant-chu-expansion}, whenever
$\ell\neq0$ and $\ell\neq k$ we obtain
\begin{equation}
 |A_{b_N}(k,\ell)|
 =\sqrt N\left|1+\tau_N^2\omega_N^{\ell(k-\ell)}\right|.
 \label{eq:off-line-chu}
\end{equation}
If $N$ is odd, $-1$ lies halfway between two $N$th roots of unity at
angular distance $\pi/N$; if $N$ is even, the factor
$\tau_N^2=e^{\pi i/N}$ shifts the $N$th-root grid by half a mesh.  In both
cases the minimum in \eqref{eq:off-line-chu} is therefore
\[
 2\sqrt N\sin\frac{\pi}{2N}.
\]
It is attained, for example, by taking $\ell=1$ and choosing $k$ so that
$\ell(k-\ell)$ is a nearest residue to the antipodal phase.

For $k\neq0$ one also obtains
\[
 |A_{b_N}(k,0)|=|A_{b_N}(k,k)|
 =N+2\sqrt N\,\operatorname{Re}\tau_N,
\]
which is larger than the off-line minimum.  Finally,
\[
 \norm{b_N}_2^2=A_{b_N}(0,0)
 =2N+2\sqrt N\,\operatorname{Re}\tau_N.
\]
This proves \eqref{eq:constant-chu-exact}.  Since
$\sin(\pi/(2N))\ge1/N$ and
$N+\sqrt N\operatorname{Re}\tau_N\le2N$, we get
\eqref{eq:universal-seed-lower}; the asymptotic follows immediately.
\end{proof}

\begin{corollary}[Universal simultaneous construction]
\label{cor:universal}
For every $N\ge2$ there is an explicit algebraic window $H_N\in\C^N$ such
that $\mathcal G(H_N)$ is full spark, does phase retrieval, and
\begin{equation}
 \alpha(H_N)\ge\frac{47}{162}N^{-3/2}.
 \label{eq:universal-fullspark}
\end{equation}
\end{corollary}

\begin{proof}
All coordinates of $b_N$ are algebraic.  Normalize this seed, choose
an explicit unit full-spark algebraic window $h_N$, and apply
Theorem~\ref{thm:regularization} with
$\varepsilon\le\alpha(b_N)/8$.  The ambiguity margin is scale invariant.
Nonvanishing of all ambiguity coefficients implies phase retrieval by
\cite[Theorem~2.2]{BojarovskaFlinth2016}.
\end{proof}

\subsection{A divisor-sensitive two-site repair}

The preceding construction is uniform in $N$.  A sparse repair of the Chu
chirp detects additional arithmetic structure.

\begin{theorem}[Divisor-sensitive Chu law]
\label{thm:divisor-chu}
Assume $N\ge25$, let $d\mid N$ with $1\le d\le\sqrt N$, and put
$L=N/d$ and
\[
 \phi_d=\frac\pi2+\frac{\pi}{4d}.
\]
Define
\begin{equation}
 b_{N,d}=q_N+\sqrt N\,e_0+
 \sqrt N\,e^{-i\phi_d}q_N(d)e_d.
 \label{eq:divisor-seed}
\end{equation}
Then
\begin{equation}
 \frac18\frac{d}{N^{3/2}}
 \le \alpha(b_{N,d})
 \le 4\pi\frac{d}{N^{3/2}}.
 \label{eq:divisor-two-sided}
\end{equation}
\end{theorem}

\begin{proof}
Let $R=\sqrt N$.  Use the shift factor $a_k$ from
\eqref{eq:chu-shift}, put $m=\ell-k$, and define
\[
 z=\omega_N^{md}=e^{2\pi i m/L},\qquad
 \xi=\omega_N^{km}.
\]
If $k\notin\{0,\pm d\}$, the two repair coordinates do not overlap under
translation by $k$.  The chirp--repair and repair--chirp contributions are,
respectively,
\[
 R a_k\xi\bigl(1+e^{-i\phi_d}z\bigr)
 \quad\text{and}\quad
 R a_k\bigl(1+e^{i\phi_d}z\bigr).
\]
Together with $A_{q_N}(k,k+m)=Na_k\mathbf1_{m=0}$, this gives the exact identity
\begin{equation}
 A_{b_{N,d}}(k,k+m)
 =a_k\left[N\mathbf1_{m=0}+R(X+\xi Y)\right],
 \label{eq:divisor-master}
\end{equation}
where
\[
 X=1+e^{i\phi_d}z,
 \qquad
 Y=1+e^{-i\phi_d}z.
\]
For $m=0$, the term in brackets is larger than $N$ in modulus.  Suppose
$m\neq0$.  If $z\neq\pm1$, then
\[
 |X+\xi Y|\ge\bigl||X|-|Y|\bigr|
 =\frac{4|\sin\phi_d\sin\theta|}{|X|+|Y|}
 \ge |\sin\phi_d\sin\theta|,
\]
where $z=e^{i\theta}$.  Since
$\sin\phi_d\ge2^{-1/2}$ and every $L$th root different from $\pm1$
satisfies $|\sin\theta|\ge\sin(\pi/L)\ge2/L$, we obtain
\begin{equation}
 |X+\xi Y|\ge\frac{\sqrt2}{L}.
 \label{eq:generic-root-gap}
\end{equation}

If $z=1$, then $m=tL$ for some integer $t$, and hence
$\xi=e^{2\pi i kt/d}$ lies on the $d$th-root grid.  The dangerous phase is
$-e^{i\phi_d}$, whose angle is $3\pi/2+\pi/(4d)$.  After scaling angles
by $d/(2\pi)$, its position is
\[
 \frac{3d}{4}+\frac18.
\]
A check of $d$ modulo $4$ shows that its distance from the nearest integer is
at least $1/8$; equivalently, $-e^{i\phi_d}$ is at angular distance at least
$\pi/(4d)$ from every $d$th root of unity.  Since
\[
 |X+\xi Y|
 =2\cos(\phi_d/2)\,|e^{i\phi_d}+\xi|,
\]
we obtain
\[
 |X+\xi Y|
 \ge4\cos(\phi_d/2)\sin\frac{\pi}{8d}
 \ge\frac{2-\sqrt2}{d}
 \ge\frac{2-\sqrt2}{L}.
 \label{eq:z-one-gap}
\]
Here we used
$\cos(\phi_d/2)\ge\sin(\pi/8)$ and concavity of sine.  If $z=-1$, then
$L$ is even and $m=(2t+1)L/2$ for some integer $t$.  Consequently
$\xi=e^{\pi i k(2t+1)/d}$ lies on the $2d$th-root grid.  The dangerous
cancellation phase is $e^{i\phi_d}$.  Scaling angles by $d/\pi$, its
position is
\[
 \frac d2+\frac14,
\]
which has distance exactly $1/4$ from the nearest integer, independently of
the parity of $d$.  Thus its angular distance from the $2d$th-root grid is
$\pi/(4d)$, and
\[
 |X+\xi Y|
 =2\sin(\phi_d/2)|\xi-e^{i\phi_d}|
 \ge4\sin(\phi_d/2)\sin\frac{\pi}{8d}
 \ge\frac{2-\sqrt2}{d}.
\]  Thus, for every
$k\notin\{0,\pm d\}$,
\begin{equation}
 |A_{b_{N,d}}(k,\ell)|
 \ge (2-\sqrt2)\frac{\sqrt N}{L}.
 \label{eq:generic-lower-A}
\end{equation}

It remains to treat the three exceptional translations.  At $(k,\ell)=(0,0)$,
the ambiguity coefficient is $\norm{b_{N,d}}_2^2$ and is therefore harmless.
At $k=0$ and $\ell\neq0$,
\[
 A_{b_{N,d}}(0,\ell)
 =(N+2R)+(N+2R\cos\phi_d)\omega_N^{\ell d},
\]
so the reverse triangle inequality gives
$|A_{b_{N,d}}(0,\ell)|\ge2R$.  At $k=d$, the repair--repair overlap has
modulus $N$.  If $m\neq0$, the mixed terms have total modulus at most
$4R$, hence
\[
 |A_{b_{N,d}}(d,d+m)|\ge N-4R\ge R.
\]
If $m=0$, then after removing the unimodular factor $a_d$ the coefficient
is
\[
 N(1+e^{i\phi_d})+2R(1+\cos\phi_d),
\]
whose imaginary part has modulus
$N\sin\phi_d\ge N/\sqrt2$.  The case $k=-d$ follows from the exact adjoint identity
\[
 A_{b_{N,d}}(-k,-\ell)
 =\omega_N^{k\ell}\overline{A_{b_{N,d}}(k,\ell)},
\]
which follows directly from \eqref{eq:ambiguity} and reduces the modulus
estimate to the case $k=d$.

Finally,
\begin{equation}
 \norm{b_{N,d}}_2^2
 =3N+2\sqrt N(1+\cos\phi_d).
 \label{eq:divisor-norm}
\end{equation}
Since $N\ge25$, this is at most $17N/5$.  Combining
\eqref{eq:generic-lower-A} with $L=N/d$ yields
\[
 \alpha(b_{N,d})
 \ge \frac{5(2-\sqrt2)}{17}\frac{d}{N^{3/2}}
 >\frac18\frac{d}{N^{3/2}}.
\]

For the upper bound take $m=1$, so $z=e^{2\pi i/L}$.  The continuous
minimum of $|X+\xi Y|$ over $|\xi|=1$ equals
$\bigl||X|-|Y|\bigr|$.  Because $|X|=|Y|$ at $z=1$ and both depend
Lipschitz-continuously on $z$,
\[
 \bigl||X|-|Y|\bigr|\le2|z-1|\le\frac{4\pi}{L}.
\]
Among the four $N$th roots closest to any prescribed target phase, at
least one remains after deleting the three forbidden values $k=0,\pm d$.
Hence some allowed $\omega_N^k$ lies within angular distance at most
$4\pi/N$ of a continuous minimizer.  Since $|Y|\le2$, for this allowed
$k$ we have
\[
 |X+\omega_N^kY|\le\frac{4\pi}{L}+\frac{8\pi}{N}
 \le\frac{12\pi}{L}.
\]
Using \eqref{eq:divisor-master} and
$\norm{b_{N,d}}_2^2\ge3N$ gives
\[
 \alpha(b_{N,d})\le4\pi\frac{d}{N^{3/2}},
\]
as required.
\end{proof}

\begin{corollary}[Arithmetic refinement after regularization]
\label{cor:divisor-fullspark}
For $N\ge25$, let
\[
 d_*(N)=\max\{d:d\mid N,\ d\le\sqrt N\}.
\]
There is an explicit algebraic full-spark phase-retrieval window $G_N$ with
\begin{equation}
 \alpha(G_N)\ge\frac{47}{648}
 \frac{d_*(N)}{N^{3/2}}.
 \label{eq:divisor-regularized}
\end{equation}
In particular, if $N$ is a square, the regularized construction satisfies
$\alpha(G_N)\ge cN^{-1}$ with an absolute constant $c>0$.
\end{corollary}

\begin{proof}
The coordinates of $b_{N,d_*(N)}$ are algebraic.  Normalize this seed and apply
Theorem~\ref{thm:regularization} with
$\varepsilon\le\alpha(b_{N,d_*(N)})/8$.
\end{proof}

\subsection{Chinese-remainder factorization}

\begin{theorem}[Exact CRT multiplicativity]
\label{thm:crt}
Let
\[
 N=n_1\cdots n_s,
 \qquad \gcd(n_i,n_j)=1\quad(i\neq j).
\]
Under the Chinese-remainder unitary identification
\[
 \C^{\Z_N}\simeq\bigotimes_{i=1}^s\C^{\Z_{n_i}},
\]
let $v_i\in\C^{\Z_{n_i}}$ be nonzero and set
$v=v_1\otimes\cdots\otimes v_s$.  Then
\begin{equation}
 \alpha(v)=\prod_{i=1}^s\alpha(v_i).
 \label{eq:crt-alpha}
\end{equation}
If all $v_i$ are algebraic and have positive ambiguity margin, a single
application of Theorem~\ref{thm:regularization} produces an explicit
full-spark phase-retrieval window $G_N$ satisfying
\begin{equation}
 \alpha(G_N)\ge\frac{47}{81}
 \prod_{i=1}^s\alpha(v_i).
 \label{eq:crt-regularized}
\end{equation}
\end{theorem}

\begin{proof}
Put $N_i=N/n_i$ and choose $s_i$ with
$N_is_i\equiv1\pmod{n_i}$.  The CRT representation
$j\equiv\sum_i j_iN_is_i\pmod N$ gives
\[
 T_k\longmapsto\bigotimes_iT_{k_i},
 \qquad
 M_\ell\longmapsto\bigotimes_iM_{\ell_i},
\]
where $k_i\equiv k\pmod{n_i}$ and
$\ell_i\equiv\ell s_i\pmod{n_i}$.  Hence
\[
 A_v(k,\ell)=\prod_iA_{v_i}(k_i,\ell_i),
 \qquad
 \norm v_2^2=\prod_i\norm{v_i}_2^2.
\]
Both $k\mapsto(k_i)_i$ and
$\ell\mapsto(\ell_i)_i$ are bijections.  The local minimizing phase-space
indices may therefore be realized simultaneously, proving the exact equality
\eqref{eq:crt-alpha}.  Since $\alpha$ is invariant under nonzero scalar multiplication, normalize each
$v_i$ to unit norm; algebraicity is preserved under this normalization, and the
tensor seed is unit with the same ambiguity margin.  Choose any explicit algebraic full-spark auxiliary window and normalize
it as well.  Now choose the algebraic perturbation parameter in
Theorem~\ref{thm:regularization} to be at most
$\frac18\prod_i\alpha(v_i)$ and apply the theorem once to the normalized
tensor seed.
\end{proof}

\subsection{Optimal squarefree scale and near-SIC prime regularization}

We now use a recent prime-dimensional input.  Zhu and Wang construct, for
every prime power $q=p^r$ of characteristic $p\ge5$, a balanced-Alltop unit
vector $\phi_q$ (building on Alltop's cubic phase construction
\cite{Alltop1980}) whose nonidentity projector-Gram spectrum lies in an explicit
interval $[L_q,U_q]$ and whose minimum equals $L_q$
\cite[Theorem~11]{ZhuWang2026}.  For prime $q=p$, this is the same cyclic
phase space used here; for $r>1$ their finite-field phase space is different
from $\Z_{p^r}^2$, as they explicitly emphasize.  Their formulas give
\begin{equation}
 L_p\le p|A_{\phi_p}(k,\ell)|^2\le U_p
 \qquad((k,\ell)\neq(0,0)),
 \label{eq:ZW-prime}
\end{equation}
with
\begin{equation}
 L_p\ge L_5=0.197863708777\ldots,
 \qquad
 \frac{U_p}{L_p}\longrightarrow1.
 \label{eq:ZW-flat}
\end{equation}
The seed $\phi_p$ is algebraic: its Alltop coordinates are roots of unity
and the balancing parameter
\[
 t_p=\frac1{\sqrt{4+\sqrt p}+2}
\]
is algebraic.

\begin{corollary}[Squarefree optimal order]
\label{cor:squarefree}
Let
\[
 N=p_1\cdots p_s
\]
be squarefree with distinct primes $p_i\ge5$.  Then there is an explicit
algebraic full-spark phase-retrieval window $G_N$ satisfying
\begin{equation}
 \alpha(G_N)\ge\frac{47}{81}
 \frac{L_5^{s/2}}{\sqrt N}.
 \label{eq:squarefree-new}
\end{equation}
Consequently, for every fixed $s_0$, this is the optimal order
$N^{-1/2}$ uniformly over such $N$ with $s\le s_0$.
\end{corollary}

\begin{proof}
Take the tensor product of the prime balanced-Alltop seeds.  Equations
\eqref{eq:ZW-prime}--\eqref{eq:ZW-flat} and
Theorem~\ref{thm:crt} give a unit tensor seed with ambiguity margin at least
$L_5^{s/2}/\sqrt N$.  Apply one algebraic regularization.  Optimality of the
exponent follows from Lemma~\ref{lem:moyal}.
\end{proof}

\begin{corollary}[Full spark with asymptotically SIC-flat prime spectrum]
\label{cor:near-sic-fullspark}
For every prime $p\ge5$ there is an explicit algebraic unit vector
$\widetilde g_p$ such that $\mathcal G(\widetilde g_p)$ is full spark and
does phase retrieval.  Its nonidentity projector-Gram spectrum is contained
in
\begin{equation}
 [L_p-Cp^{-3/2},\ U_p+Cp^{-3/2}]
 \label{eq:near-sic-interval}
\end{equation}
for an absolute constant $C$.  Let $\mathbf 1$ denote the constant vector in
phase-space coefficient space and let $\kappa$ denote the spectral condition
number.  Then, in particular,
\begin{equation}
 p\,\alpha(\widetilde g_p)^2\longrightarrow1,
 \qquad
 \kappa\!\left(G^{\Pi}_{\widetilde g_p}\big|_{\mathbf 1^\perp}\right)\longrightarrow1.
 \label{eq:near-sic-limits}
\end{equation}
\end{corollary}

\begin{proof}
Let $\phi_p$ be the balanced-Alltop unit vector above, and let $h_p$ be an
explicit algebraic unit full-spark window.  Choose the algebraic parameter in
Theorem~\ref{thm:regularization} so that
\[
 \varepsilon_p\le
 \min\left\{p^{-2},\frac{\sqrt{L_5/p}}8\right\}.
\]
Put $g_p=\phi_p+\varepsilon_ph_p$ and
$\widetilde g_p=g_p/\norm{g_p}_2$.  Theorem~\ref{thm:regularization} gives
full spark and a positive ambiguity floor.

Write $\delta_p=2\varepsilon_p+\varepsilon_p^2$.  Since
$\norm{g_p}_2^2\ge(1-\varepsilon_p)^2$ and
$\varepsilon_p\le1/4$,
\begin{equation}
 |A_{\widetilde g_p}(u)-A_{\phi_p}(u)|
 \le8\varepsilon_p
 \qquad(u\in\Z_p^2).
 \label{eq:normalized-perturbation}
\end{equation}
For $u\neq0$, Theorem~11 and the explicit formulas in
\cite{ZhuWang2026} give a uniform bound $U_p=O(1)$ for $p\ge5$.  Therefore
\[
 p\left||A_{\widetilde g_p}(u)|^2-|A_{\phi_p}(u)|^2\right|
 \le C p^{-3/2}
\]
with an absolute $C$.  Proposition~\ref{prop:diagonalization} and
\eqref{eq:ZW-prime} give \eqref{eq:near-sic-interval}.
Finally $L_p,U_p\to1$ by \cite{ZhuWang2026}.  Hence $U_p<p$ for
all sufficiently large primes $p$.  For such $p$, every nonidentity ambiguity
coefficient of $\phi_p$ has modulus strictly below one; the same remains true
after the $O(p^{-2})$ regularized perturbation.  Thus, for all sufficiently
large $p$, the minimum defining $\alpha$ is attained off the identity.  Together
with \eqref{eq:near-sic-interval} and $L_p,U_p\to1$, this proves
\eqref{eq:near-sic-limits}.
\end{proof}

\section{Stability and representation consequences}
\label{sec:consequences}

\subsection{Lifted stability}

\begin{corollary}[Deterministic lifted inversion]
\label{cor:lifted}
Let $g$ be any of the unit-normalized regularized windows above.  Then for
all $X\in\C^{N\times N}$,
\begin{equation}
 \norm X_{\mathrm{HS}}
 \le \frac1{\sqrt N\,\alpha(g)}
 \norm{\mathcal Q_g(X)}_2.
 \label{eq:inverse-lifted}
\end{equation}
For unit vectors $x,y$ this yields
\begin{equation}
 \inf_{|\tau|=1}\norm{x-\tau y}_2
 \le \frac1{\sqrt N\,\alpha(g)}
 \norm{\mathcal Q_g(xx^*)-\mathcal Q_g(yy^*)}_2.
 \label{eq:signal-stability}
\end{equation}
\end{corollary}

\begin{proof}
The first estimate is Proposition~\ref{prop:diagonalization}.  For unit
vectors,
\[
 \norm{xx^*-yy^*}_{\mathrm{HS}}^2
 =2(1-|\ip{x}{y}|^2)
 \ge2(1-|\ip{x}{y}|)
 =\inf_{|\tau|=1}\norm{x-\tau y}_2^2.
\]
Apply \eqref{eq:inverse-lifted} to $X=xx^*-yy^*$.
\end{proof}

\begin{remark}[Conditioning scales]
The universal construction in Corollary~\ref{cor:universal} gives inverse
constant $O(N)$ in \eqref{eq:inverse-lifted}.  The squarefree family in
Corollary~\ref{cor:squarefree}, with a bounded number of prime factors, gives
a dimension-independent inverse constant.  For the prime family in
Corollary~\ref{cor:near-sic-fullspark}, the squared singular values of
$\mathcal Q_{\widetilde g_p}$ on the traceless Weyl sector lie in the interval
\eqref{eq:near-sic-interval}.  Hence, for all sufficiently large $p$, its traceless-sector condition number is
\[
 \kappa_{\rm tr}(\mathcal Q_{\widetilde g_p})
 \le
 \sqrt{\frac{U_p+Cp^{-3/2}}{L_p-Cp^{-3/2}}}
 \longrightarrow1.
\]
The corresponding projector-Gram condition number is the square of this
quantity.
\end{remark}

\begin{remark}[Scope of the stability statement]
\label{rem:stability-scope}
The estimates above concern the complete $N^2$-measurement lifted operator
$\mathcal Q_g$.  Full spark is used here as a qualitative maximal
linear-independence property: we do not claim a uniform lower singular-value
bound for every $N\times N$ Gabor submatrix, nor a stability estimate for
arbitrary erasures of the Gabor orbit.  For general connections between full
spark and robustness to erasures, see \cite{Li2026}.
\end{remark}

\subsection{Schr\"odinger representatives}

Let $\pi_N$ denote the standard Schr\"odinger representation of the finite
Heisenberg group $\mathbb H_N$.  Its center acts by scalars and is therefore
contained in the projective kernel.  Following F\"uhr and Oussa, we use the
standard transversal
\[
 W=\{(k,\ell,0):k,\ell\in\Z_N\}
\]
of the center; the corresponding representative family is precisely the
Gabor system $\{M_\ell T_kg:(k,\ell)\in\Z_N^2\}$.

\begin{corollary}[Explicit simultaneous Schr\"odinger representatives]
\label{cor:schrodinger}
For every $N\ge2$ there is an explicit algebraic vector $g\in\C^N$ such
that the representative family $(\pi_N(w)g)_{w\in W}$ is full spark and
phase retrievable.  Moreover every nonzero $x\in\C^N$ has at most $N-1$
zeros among the $N^2$ representative coefficients
\[
 \ip{x}{M_\ell T_kg},
 \qquad (k,\ell)\in\Z_N^2,
\]
and this bound is attained.  Consequently
\[
 p_0(\pi_N,g)=\frac{N-1}{N^2},
\]
which is the sharp value $p_0(\pi_N)$ identified by F\"uhr and Oussa.
\end{corollary}

\begin{proof}
Take the window from Corollary~\ref{cor:universal}.  Full spark and phase
retrieval hold for the representative Gabor family.  If a nonzero vector
$x$ were orthogonal to $N$ representatives, those $N$ vectors would be
linearly dependent, a contradiction.  Conversely, any $N-1$ representatives
are independent, so their orthogonal complement contains a nonzero vector;
full spark prevents that vector from being orthogonal to any additional
representative.  Hence the maximal zero proportion on $W$ is
$(N-1)/N^2$.  Remark~17 and Remark~18 of \cite{FuehrOussa2023} identify
this representative count with $p_0(\pi_N,g)$ and give the sharp value
$p_0(\pi_N)=(N-1)/N^2$.
\end{proof}

The distinction between the full Heisenberg group orbit and the representative
system is essential: central elements act by scalars, so the full group-indexed
orbit contains proportional repetitions and cannot be full spark in the sense
used in this paper.  F\"uhr and Oussa formulate their finite Schr\"odinger
full-spark discussion on such a system of representatives modulo the center
\cite[Remarks~17--18 and Concluding Remarks]{FuehrOussa2023}.  Thus
Corollary~\ref{cor:schrodinger} gives the explicit simultaneous
phase-retrieval/full-spark representative vector requested there and supplies
the sharp zero proportion used in their split-$K$-triple induction.

\medskip
The algebraic regularization theorem separates the determinant constraint
from the analytic or arithmetic problem of constructing a good ambiguity
seed.  The Chu seeds quantify what is available in every cyclic dimension and
how a balanced divisor improves the margin; CRT multiplicativity transfers
strong local seeds to composite dimensions; and near-SIC prime seeds remain
near-SIC after full spark is imposed.  Together these results give a single
deterministic route from quantitative ambiguity control to explicit
full-spark phase-retrieval Gabor windows.

\section*{Data availability}
No datasets were generated or analyzed in this study.

\section*{AI use statement}
OpenAI ChatGPT was used for literature search, proof verification, numerical
verification, and language polishing.  The author takes full responsibility for all
mathematical claims and the final manuscript.

\end{document}